\documentclass{article}%
\usepackage{amsmath}%
\usepackage{amsfonts}%
\usepackage{amssymb}%
\usepackage{graphicx}
\newtheorem{theorem}{Theorem}

\newtheorem{definition}[theorem]{Definition}
\newtheorem{example}[theorem]{Example}

\newtheorem{proposition}[theorem]{Proposition}

\newenvironment{proof}[1][Proof]{\textbf{#1.} }{\ \rule{0.5em}{0.5em}}

\begin{document}

\title{ON THE CHOLESKY METHOD }
\author{Christian Rakotonirina
\\Institut Sup\'{e}rieur de Technologie d'Antananarivo, IST-T, BP 8122, Madagascar\\
e-mail: rakotopierre@refer.mg}
\maketitle

\begin{abstract}
In this paper, we prove that if the matrix of the linear system is symmetric, the Cholesky decomposition can be obtained from  the Gauss elimination method without pivoting, without proving that the matrix of the system is positive definite. \end{abstract}


\section{Introduction}
\noindent The direct methods for solving linear system , Gauss elimination method, LU decomposition and Cholesky method are well known. We agreed with some authors \cite{DemidovitchMaron79},\cite{Nougier87} that the LU decomposition and the Cholesky method are helpful for solving many linear systems of the same matrix, whose difference is only the constants at the right hand side.

The Gauss elimination method with or without pivoting can lead us to the LU decomposition. The Gauss elimination method with pivoting applied to linear system with symmetric matrix can't lead us to the Cholesky decomposition because pivoting alters the symmetry. However, sometimes we don't need pivoting.

Since the Gauss elimination can lead us to the LU decomposition, we think that it is better to solve at first one of these linear systems by the Gauss elimination and the others by LU decomposition. So, we will be allowed to solve at first one of these linear systems by the Gauss eliminaton and the others by Cholesky method, in the case where the matrix is symmetric positive definite, if Gauss elimination can lead us to the Cholesky method.

In this paper, we will see that there exist relation between Gauss elimination without pivoting and the Cholesky method. That is Gauss elimination without pivoting can lead us to Cholesky decomposition. We will prove this relation by using the LU decomposition.

We organize the paper as the following. In the second section we represent the LU decomposition by using the Gauss elimination method. In the third section we will prove,  with the help of the LU decomposition we can have the Cholesky decomposition from the Gauss elimination method. At the last section an example will be tread for clarify the method.

\section{Gauss elimination}

\noindent Let us consider the following system of linear n equations and n unknowns\\

$\left\{\begin{array}{clrrrrrr} %
a_{11}x_1+a_{12}x_2+\ldots+a_{1n}x_n & =  b_1\\
a_{21}x_1+a_{22}x_2+\ldots+a_{2n}x_n & =  b_2\\
\ldots\ldots\ldots\ldots\ldots\ldots\ldots\ldots\ldots\\
a_{i1}x_1+a_{i2}x_2+\ldots+a_{in}x_n & =  b_i\\
\ldots\ldots\ldots\ldots\ldots\ldots\ldots\ldots\ldots\\
a_{n1}x_1+a_{n2}x_2+\ldots+a_{nn}x_n & =  b_n
\end{array}\right.$\\
\noindent which can be written under matricial form
\begin{equation}\label{eqn1}
AX=B
\end{equation}\\
where\\
$A=\left(a_{ij}\right)_{1\leq i,j\leq n}$,
$X=\begin{pmatrix}
x_1\\
	x_2\\
	\vdots\\
	x_n
\end{pmatrix}$,
$B=\begin{pmatrix}
	b_1\\
	b_2\\
	\vdots\\
	b_n
\end{pmatrix}$
\\
\noindent If $a_{11}\neq0$, Gauss elimination in the first column is written\\

$\left\{\begin{array}{clrrrrrr} %
a_{11}^{(0)}x_1 & + & a_{12}^{(0)}x_2 & + & \ldots & + & a_{1n}^{(0)}x_n & =  b_{1}^{(0)}\\
\ & \ & a_{22}^{(1)}x_2 & + & \ldots & + & a_{2n}^{(1)}x_n & =  b_{2}^{(1)}\\
\ & \  & \ldots & \ldots & \ldots & \ldots & \ldots & \ldots  \ldots\\
\ & \ & a_{i2}^{(1)}x_2 & + & \ldots & + & a_{in}^{(1)}x_n & =  b_{i}^{(1)}\\
\ & \ & \ldots & \ldots & \ldots & \ldots & \ldots & \ldots  \ldots\\
\ & \ & a_{n2}^{(1)}x_2 & + & \ldots & + & a_{nn}^{(1)}x_n & =  b_{n}^{(1)}
\end{array}\right.$

\noindent with $a_{1j}^{(0)}=a_{1j}$ and $a_{ij}^{(1)}=-\frac{a_{i1}a_{1j}-a_{11}a_{ij}}{a_{11}}$,\\

 \noindent which may matricially be written
 \begin{equation}\label{eqn2}
A_1X=B_1
\end{equation}\\
and can be obtained in multiplying (\ref{eqn1}) by the lower triangular matrix\cite{KincaidCheney99}, \cite{Steven04} \\
\begin{equation}\nonumber
G_1=\left(\begin{array}{clrrrrr} %
	1 & 0 & 0  & \ldots & 0\\
	\  & \  & \   & \  & \ \\
	-\frac{a_{21}^{(0)}}{a_{11}^{(0)}} & 1 & 0 & \ldots & 0\\
	-\frac{a_{31}^{(0)}}{a_{11}^{(0)}} & 0 & 1 & \ldots & 0\\
	\vdots & \vdots & \vdots & \ddots & \vdots\\
-\frac{a_{n1}^{(0)}}{a_{11}^{(0)}} & 0 & 0 & \ldots & 1	
\end{array}\right)
\end{equation}

 $A_1=G_1A$ and $B_1=G_1B$.\\

 The Gauss elimination at the second column can be obtained by multiplying to the relation (\ref{eqn2}) the lower triangular matrix\\
\begin{equation}\nonumber
 G_2=\left(\begin{array}{clrrrrrr} %
	1 & 0 & 0  & 0 & \ldots & 0\\
	0 & 1 & 0  & 0 & \ldots & 0\\
	\  & \  & \   & \  & \  & \ \\
	  0 & -\frac{a_{32}^{(1)}}{a_{22}^{(1)}} & 1 & 0 & \ldots & 0\\
0 & -\frac{a_{42}^{(1)}}{a_{22}^{(1)}}& 0 & 1 & \ldots & 0\\
	\vdots & \vdots & \vdots & \vdots & \ddots & \vdots\\
0 & -\frac{a_{n2}^{(1)}}{a_{22}^{(1)}} & 0 & 0 & \ldots & 1	
\end{array}\right)
\end{equation}\\

The matricial equation (\ref{eqn2}) becomes
\begin{equation}\nonumber
A_2X=B_2
\end{equation}\\
with $A_2=G_2G_1A$ and $B_2=G_2G_1B$.\\

In continuing so, we have finally,
\begin{equation}\label{eqn3}
UX=B'
\end{equation}\\
with $U=G_{n-1}\ldots G_2G_1A$ an upper triangular matrix and $B'=G_{n-1}\ldots G_2G_1B$.\\
We can check easily that the inverse of $G_l$ is \\
\begin{equation}\nonumber
G^{-1}_l=\left(\begin{array}{clrrrrrrr} %
	1 & \ & 0 & \ldots   & 0 &  \ & \ldots &  \ & 0\\
	\  & \ & \  & \   & \  & \  & \ & \ & \ \\	
	0 & \ & 1 & \ldots   & 0 & \ & \ldots & \ & 0\\
	\  & \ & \  & \   & \  & \  & \ & \ & \ \\	
		\vdots & \ & \vdots & \ddots  & \vdots & \  & \ & \ & \vdots\\
		\  & \ & \  & \   & \  & \  & \ & \ & \ \\		
	  0 & \ & 0 & \ldots & 1 & \ & \ldots & \ & 0\\
	  \  & \ & \  & \   & \  & \  & \ & \ & \ \\	
0 & \ & 0 & \ldots & \frac{a_{(l+1)l}^{(l-1)}}{a_{ll}^{(l-1)}} & \ & \ldots & \ & 0\\
	\vdots & \ & \vdots & \  & \vdots & \ & \ddots & \ & \vdots\\
0 & \ & 0 & \ldots & \frac{a_{nl}^{(l-1)}}{a_{ll}^{(l-1)}} & \ & \ldots & \ & 1\\
\end{array}\right)
\end{equation}
So, $G_{n-1}G_{n-2}\cdots G_2G_1$ is an invertible matrix and the equation (\ref{eqn3}) becomes
\begin{equation}\nonumber
LUX=B
\end{equation}\\
\noindent with $L=G^{-1}_1G^{-1}_2\cdots G^{-1}_{n-2}G^{-1}_{n-1}$ and we have the decomposition of $A$ as product of lower triangular matrix by an upper triangular matrix
\begin{equation}\nonumber
A=LU
\end{equation}\\
One can check easily that
\begin{equation}\nonumber
L=\left(\begin{array}{clrrrrrrr} %
	1 & \ & 0 & \ldots   & 0 &  \ & \ldots &  \ & 0\\
	\  & \ & \  & \   & \  & \  & \ & \ & \ \\	
	\frac{a_{21}^{(0)}}{a_{11}^{(0)}} & \ & 1 & \ldots   & 0 & \ & \ldots & \ & 0\\
	\  & \ & \  & \   & \  & \  & \ & \ & \ \\	
		\frac{a_{31}^{(0)}}{a_{11}^{(0)}} & \ & \frac{a_{32}^{(1)}}{a_{22}^{(1)}} & \ddots  & \vdots & \  & \ & \ & \vdots\\
		\  & \ & \  & \   & \  & \  & \ & \ & \ \\		
	  \vdots & \ & \vdots & \ldots & 1 & \ & \ldots & \ & 0\\
	  \  & \ & \  & \   & \  & \  & \ & \ & \ \\	
\frac{a_{(l+1)1}^{(0)}}{a_{11}^{(0)}} & \ & \frac{a_{(l+1)2}^{(1)}}{a_{22}^{(1)}} & \ldots & \frac{a_{(l+1)l}^{(l-1)}}{a_{ll}^{(l-1)}} & \ & \ddots & \ & 0\\
	\vdots & \ & \vdots & \  & \vdots & \ & \  & \ & \vdots\\
\frac{a_{n1}^{(0)}}{a_{11}^{(0)}} & \ & \frac{a_{(n)2}^{(1)}}{a_{22}^{(1)}} & \ldots & \frac{a_{nl}^{(l-1)}}{a_{ll}^{(l-1)}} & \ & \ldots & \ & 1\\
\end{array}\right)
\end{equation}\\
\section{Cholesky method}
Now suppose $A$ is a symmetric positive definite matrix. Then, the Cholesky method consists to decompose $A$ as the product
\begin{equation}\nonumber
A=G^TG
\end{equation}\\
with $G$ is a upper triangular matrix and $G^T$ his transpose.

Let us at first generalize this decomposition.
\begin{definition}
Let $A=\left(a_{ij}\right)_{1\leq i,j\leq n}$ a complex symmetric matrix. Let us call Gauss matrix of $A$ the upper triangular matrix $U(A)$, obtained after transforming $A$ by the Gauss elimination above.
\end{definition}
\begin{proposition}\label{prop}
Let $A=\left(a_{ij}\right)_{1\leq i,j\leq n}$ a complex symmetric matrix, such that $det(A)\neq 0$,
\begin{equation}\nonumber
U(A)=\left(\begin{array}{clrrrrr} %
	u_{11} & u_{12} & \ldots & u_{1n}\\
	
	0 & u_{22} & \ldots & u_{2n}\\
	
	\vdots & \vdots & \ddots &  \vdots\\
0 & 0 & \ldots & u_{nn}	
\end{array}\right)
\end{equation}\\
the Gauss matrix of $A$. Then $A$ can be decomposed as the product $A=G^TG$ with
 \begin{equation}\nonumber
G=\left(\begin{array}{clrrrrr} %
	\frac{u_{11}}{\sqrt{u_{11}}} & \frac{u_{12}}{\sqrt{u_{11}}} & \ldots & \frac{u_{1n}}{\sqrt{u_{11}}}\\
	
	0 & \frac{u_{22}}{\sqrt{u_{22}}} & \ldots & \frac{u_{2n}}{\sqrt{u_{22}}}\\
	
	\vdots & \vdots & \ddots &  \vdots\\
0 & 0 & \ldots & \frac{u_{nn}}{\sqrt{u_{nn}}}	
\end{array}\right)
\end{equation}\\
with $\sqrt{u_{ii}}$ a root squared of the complex number $u_{ii}$.
\end{proposition}
\begin{proof}
\begin{equation}\nonumber
U(A)=\left(\begin{array}{clrrrrr} %
	a^{(0)}_{11} & a^{(0)}_{12} & \ldots & a^{(0)}_{1n}\\
	
	0 & a^{(1)}_{22} & \ldots & a^{(1)}_{2n}\\
	
	\vdots & \vdots & \ddots &  \vdots\\
0 & 0 & \ldots & a^{(n-1)}_{nn}	
\end{array}\right)
\end{equation}\\
\begin{equation}\nonumber
L(A)=\left(\begin{array}{clrrrrrrr} %
	1 & \ & 0 & \ldots   & 0 &  \ & \ldots &  \ & 0\\
	\  & \ & \  & \   & \  & \  & \ & \ & \ \\	
	\frac{a_{21}^{(0)}}{a_{11}^{(0)}} & \ & 1 & \ldots   & 0 & \ & \ldots & \ & 0\\
	\  & \ & \  & \   & \  & \  & \ & \ & \ \\	
		\frac{a_{31}^{(0)}}{a_{11}^{(0)}} & \ & \frac{a_{32}^{(1)}}{a_{22}^{(1)}} & \ddots  & \vdots & \  & \ & \ & \vdots\\
		\  & \ & \  & \   & \  & \  & \ & \ & \ \\		
	  \vdots & \ & \vdots & \ldots & 1 & \ & \ldots & \ & 0\\
	  \  & \ & \  & \   & \  & \  & \ & \ & \ \\	
\frac{a_{(l+1)1}^{(0)}}{a_{11}^{(0)}} & \ & \frac{a_{(l+1)2}^{(1)}}{a_{22}^{(1)}} & \ldots & \frac{a_{(l+1)l}^{(l-1)}}{a_{ll}^{(l-1)}} & \ & \ddots & \ & 0\\
	\vdots & \ & \vdots & \  & \vdots & \ & \  & \ & \vdots\\
\frac{a_{n1}^{(0)}}{a_{11}^{(0)}} & \ & \frac{a_{(n)2}^{(1)}}{a_{22}^{(1)}} & \ldots & \frac{a_{nl}^{(l-1)}}{a_{ll}^{(l-1)}} & \ & \ldots & \ & 1\\
\end{array}\right)
\end{equation}\\
\begin{equation}\nonumber
A=L(A)D^{-1}DU(A)
\end{equation}\\
with
\begin{equation}\nonumber
D=\left(\begin{array}{clrrrrr} %
	\frac{1}{\sqrt{a^{(0)}_{11}}} & 0 & \ldots & 0\\
	
	0 & \frac{1}{\sqrt{a^{(1)}_{22}}} & \ldots & 0\\
	
	\vdots & \vdots & \ddots &  \vdots\\
0 & 0 & \ldots & \frac{1}{\sqrt{a^{(n-1)}_{nn}}}	
\end{array}\right)
\end{equation}\\
Let
$G=DU(A)$. Since $A$ symmetric, hence $L(A)D^{-1}=G^T$.

\end{proof}
\begin{example}
Consider the following two systems of linear equations\\

\begin{equation}\nonumber
\left\{\begin{array}{clrrrrrr} %
x_1  & - & x_2 & \ & \ & + & x_4 & =  3\\
-x_1 & + & 5x_2 & + & 2x_3 & - & 3x_4 & =  -5\\
\ & \ & 2x_2 & + & 5x_2 & + & x_4 & =  -7\\
x_1 & - & 3x_2 & + & x_3 & + & 4x_4 & =  2
\end{array}\right.
\end{equation}

\begin{equation}\nonumber
\left\{\begin{array}{clrrrrrr} %
x_1  & - & x_2 & \ & \ & + & x_4 & =  3\\
-x_1 & + & 5x_2 & + & 2x_3 & - & 3x_4 & =  1\\
\ & \ & 2x_2 & + & 5x_2 & + & x_4 & =  2\\
x_1 & - & 3x_2 & + & x_3 & + & 4x_4 & =  2
\end{array}\right.
\end{equation}
whose difference is only the right hand side and their matrix is symmetric.
So let us solve the first one by the Gauss elimination method, and the second one by the Cholesky method or LU method.
\begin{equation}\nonumber
\left(\begin{array}{clrrrr} %
1  & 0  & 0 & 0 \\
1 & 1 & 0 & 0 \\
0 & 0 & 1 & 0\\
-1 & 0 & 0 & 1
\end{array}\right)\left(\begin{array}{clrrrr} %
1  & -1  & 0 & 1 \\
-1 & 5 & 2 & -3 \\
0 & 2 & 5 & 1\\
1 & -3 & 1 & 4
\end{array}\right)\left(\begin{array}{clr} %
x_1   \\
x_2  \\
x_3 \\
x_4
\end{array}\right)=\left(\begin{array}{clrrrr} %
1  & 0  & 0 & 0 \\
1 & 1 & 0 & 0 \\
0 & 0 & 1 & 0\\
-1 & 0 & 0 & 1
\end{array}\right)\left(\begin{array}{clr} %
3   \\
-5  \\
-7 \\
2
\end{array}\right)
\end{equation}
\begin{equation}\nonumber
\left(\begin{array}{clrrrr} %
1  & -1  & 0 & 1 \\
0 & 4 & 2 & -2 \\
0 & 2 & 5 & 1\\
0 & -2 & 1 & 3
\end{array}\right)\left(\begin{array}{clr} %
x_1   \\
x_2  \\
x_3 \\
x_4
\end{array}\right)=\left(\begin{array}{clr} %
3   \\
-2  \\
-7 \\
-1
\end{array}\right)
\end{equation}
\begin{equation}\nonumber
\left(\begin{array}{clrrrr} %
1  & 0  & 0 & 0 \\
0 & 1 & 0 & 0 \\
0 & -\frac{1}{2} & 1 & 0\\
0 & \frac{1}{2} & 0 & 1
\end{array}\right)\left(\begin{array}{clrrrr} %
1  & -1  & 0 & 1 \\
0 & 4 & 2 & -2 \\
0 & 2 & 5 & 1\\
0 & -2 & 1 & 3
\end{array}\right)\left(\begin{array}{clr} %
x_1   \\
x_2  \\
x_3 \\
x_4
\end{array}\right)=\left(\begin{array}{clrrrr} %
1  & 0  & 0 & 0 \\
0 & 1 & 0 & 0 \\
0 & -\frac{1}{2} & 1 & 0\\
0 & \frac{1}{2} & 0 & 1
\end{array}\right)\left(\begin{array}{clr} %
3   \\
-2  \\
-7 \\
-1
\end{array}\right)
\end{equation}
\begin{equation}\nonumber
\left(\begin{array}{clrrrr} %
1  & -1  & 0 & 1 \\
0 & 4 & 2 & -2 \\
0 & 0 & 4 & 2\\
0 & 0 & 2 & 2
\end{array}\right)\left(\begin{array}{clr} %
x_1   \\
x_2  \\
x_3 \\
x_4
\end{array}\right)=\left(\begin{array}{clr} %
3   \\
-2  \\
-6 \\
-2
\end{array}\right)
\end{equation}
\begin{equation}\nonumber
\left(\begin{array}{clrrrr} %
1  & 0  & 0 & 0 \\
0 & 1 & 0 & 0 \\
0 & 0 & 1 & 0\\
0 & 0 & -\frac{1}{2} & 1
\end{array}\right)\left(\begin{array}{clrrrr} %
1  & -1  & 0 & 1 \\
0 & 4 & 2 & -2 \\
0 & 0 & 4 & 2\\
0 & 0 & 2 & 2
\end{array}\right)\left(\begin{array}{clr} %
x_1   \\
x_2  \\
x_3 \\
x_4
\end{array}\right)=\left(\begin{array}{clrrrr} %
1  & 0  & 0 & 0 \\
0 & 1 & 0 & 0 \\
0 & 0 & 1 & 0\\
0 & 0 & -\frac{1}{2} & 1
\end{array}\right)\left(\begin{array}{clr} %
3   \\
-2  \\
-6 \\
-2
\end{array}\right)
\end{equation}
\begin{equation}\nonumber
\left(\begin{array}{clrrrr} %
1  & -1  & 0 & 1 \\
0 & 4 & 2 & -2 \\
0 & 0 & 4 & 2\\
0 & 0 & 0 & 1
\end{array}\right)\left(\begin{array}{clr} %
x_1   \\
x_2  \\
x_3 \\
x_4
\end{array}\right)=\left(\begin{array}{clr} %
3   \\
-2  \\
-6 \\
1
\end{array}\right)
\end{equation}
$x_4=1$, $4x_3+2=-6$, $x_3=-2$, $4x_2-4-2=-2$, $x_2=1$, $x_1-1+1=3$, $x_1=3$

Now, let us move to the second equation, which can be written, after the Proposition~\ref{prop}, $G^TGX=B$ and solved in writing $\left\{\begin{array}{clrr} %
G^TY &=B \\
GX &=Y
\end{array}\right.$

\begin{equation}\nonumber
\left(\begin{array}{clrrrr} %
1  & 0  & 0 & 0 \\
-1 & 2 & 0 & 0 \\
0 & 1 & 2 & 0\\
1 & -1 & 1 & 1
\end{array}\right)\left(\begin{array}{clr} %
y_1   \\
y_2  \\
y_3 \\
y_4
\end{array}\right)=\left(\begin{array}{clr} %
3   \\
1  \\
2 \\
2
\end{array}\right)
\end{equation}
$y_1=3$, $3+2y_2=1$, $y_2=2$, $2+2y_3=2$, $y_3=0$, $3-2+y_4=2$, $y_4=1$
\begin{equation}\nonumber
\left(\begin{array}{clrrrr} %
1  & -1  & 0 & 1 \\
0 & 2 & 1 & -1 \\
0 & 0 & 2 & 1\\
0 & 0 & 0 & 1
\end{array}\right)\left(\begin{array}{clr} %
x_1   \\
x_2  \\
x_3 \\
x_4
\end{array}\right)=\left(\begin{array}{clr} %
3   \\
2  \\
0 \\
1
\end{array}\right)
\end{equation}
$x_4=1$, $2x_3+1=0$, $x_3=-\frac{1}{2}$, $2x_2-\frac{1}{2}-1=2$, $x_2=\frac{7}{4}$, $x_1-\frac{7}{4}+1=3$, $x_1=\frac{15}{4}$
\end{example}

\end{document}